\documentclass[12pt]{amsart}
\usepackage{amssymb}
\usepackage{amsmath}
\usepackage{enumerate}
\usepackage{tikz}
\usepackage{comment}
\title{The geometry of $\Pi$-invertible sheaves}
\author{Stephen Kwok}
\address{Department of Mathematics, University of Bologna, Bologna, Italy 40126}
\email{stephendiwen.kwok@unibo.it}

\begin{document}

\theoremstyle{plain} \newtheorem{thm}{Theorem}[section]
\theoremstyle{plain} \newtheorem*{mythm}{Theorem}
\theoremstyle{plain} \newtheorem{lem}[thm]{Lemma}
\theoremstyle{plain} \newtheorem{prop}[thm]{Proposition}
\theoremstyle{plain} \newtheorem{cor}[thm]{Corollary}
\theoremstyle{definition} \newtheorem{defn}{Definition}
\numberwithin{equation}{section} 

\begin{abstract}
Using the fact that $\Pi$-invertible sheaves can be interpreted as locally free sheaves of modules for the super skew field $\mathbb{D}$, we give a new construction of the $\Pi$-projective superspace $\mathbb{P}^n_{\Pi, B}$ over affine $k$ superschemes $B$, $k$ an algebraically closed field. We characterize morphisms into $\mathbb{P}^n_{\Pi, B}$ and give a new interpretation of the composition of $\Pi$-invertible sheaves in terms of the algebra of $\mathbb{D}$.
\end{abstract}

\maketitle

\section{Introduction}

Line bundles are key in classical algebraic geometry. Ample (resp. very ample) line bundles on schemes give morphisms to (resp. embeddings into) projective space, e.g. the Plucker embedding of a Grassmannian.

In the development of algebraic supergeometry, it turns out that line bundles are no longer so fundamental. For instance, over $\mathbb{C}$, generic super Grassmannians possess no ample line bundles (see, e.g. \cite{CCF} or \cite{ma1} for a proof of this fact for $Gr(2|2, 4|4)$), and therefore cannot be embedded as subsupermanifolds of super projective space $\mathbb{P}^{m|n}$ for any $m|n$.

Manin \cite{ma1} has suggested that a different concept, due to I.A. Skornyakov, should be a substitute for invertible sheaves in supergeometry: that of {\it $\Pi$-invertible sheaf}. These objects
are pairs $(S, \phi)$, where $S$ is a locally free sheaf of rank $1|1$ and $\phi$ is an odd endomorphism of $S$ such that $\phi^2 = 1$. Their transition functions can be reduced to $\mathbb{G}^{1|1}_m$, a nonabelian supergroup analogous to the usual multiplicative group $\mathbb{G}_m$.

Deligne \cite{Del1} has pointed out that over $\mathbb{C}$, $\mathbb{G}^{1|1}_m$ may be interpreted as the multiplicative supergroup $\mathbb{D}^*$ of the so called ``super skew field" $\mathbb{D}$, a noncommutative {\it central simple} superalgebra. This point of view continues to be valid over any algebraically closed field $k$ of characteristic not equal to $2$, and sheds considerable light on $\Pi$-projective geometry. Many of the basic constructions in $\Pi$-projective geometry become more transparent when interpreted in terms of the algebra of $\mathbb{D}$. This is the task of this current work.

The plan of the paper is as follows. In the first section we review basic material about the ``super skew field" $\mathbb{D}$. This object may be characterized as the unique (up to Brauer equivalence) central simple superalgebra over an algebraically closed field $k$, $char(k) \neq 2$. We define super Azumaya algebras, and extend $\mathbb{D}$ to a sheaf $\underline{\mathbb{D}}_B$ of super Azumaya algebras over a $k$-superscheme $B$. We prove some basic results about the structure of $\mathbb{D}$-modules in some key special cases.

In the category of affine algebraic $B$-superschemes, we then give a new construction of the $\Pi$-projective space $\mathbb{P}^n_{\Pi, B}$. Given a (right) $\mathbb{D}$-module $(V, \phi)$, we realize $\mathbb{P}_\Pi(V)$ as a quotient of $V \backslash \{0\}$ by the algebraic supergroup $\mathbb{G}^{1|1}_m = \mathbb{D}^*$. This result was asserted in \cite{Levin}, but no proof given. We prove that with our definition of the $\Pi$-projective space, $\underline{V} \backslash \{0\}$ becomes a $\mathbb{D}^*$-principal bundle over $\mathbb{P}_\Pi(V)$. This is enough to show that $\mathbb{P}_\Pi(V)$ is a quotient of $\underline{V} \backslash \{0\}$ by $\mathbb{D}^*$. Many of the basic results of supergeometry needed for this portion of the paper may be found in \cite{CCF}; a more detailed treatment of basic supercommutative algebra and algebraic supergeometry is contained in the Ph.D. thesis of Westra \cite{Wes}.

We then briefly discuss the theory of $\Pi$-invertible sheaves. In \cite{Del1} it is noted that a $\Pi$-invertible sheaf is nothing more than a locally free sheaf of $\mathbb{D}$-modules of rank $1$. We then define a ``hyperplane bundle" $\mathcal{O}_\Pi(1)$ on $\mathbb{P}_{\Pi, B}(V)$. This $\Pi$-invertible sheaf is key in the characterization of all $B$-morphisms $X \to \mathbb{P}_{\Pi, B}(V)$. 

Finally we define a product structure on the set of $\Pi$-invertible sheaves, taking values in $1|1$ locally free sheaves, using the algebra of the super skew field $\mathbb{D}$. This product is shown to be the same as the composition of $\Pi$-invertible sheaves proposed by Voronov, Manin, and Penkov \cite{VMP}.\\

\noindent {\bf Acknowledgments.} This work grew out of the author's Ph.D. thesis. The author owes much to his advisor, V.S. Varadarajan, for his guidance, patience, insights, and moral support, all of which he shared generously. The author is also greatly indebted to P. Deligne for pointing out the interpretation of $\mathbb{G}^{1|1}_m$ and the connection between $\Pi$-invertible sheaves and $\mathbb{D}$ in \cite{Del1} and \cite{Del2}, which provided the germ of the present work.

\section{The super Azumaya algebra $\mathbb{D}$}

\subsection{The super skew field $\mathbb{D}$.} 

Let $k$ denote an algebraically closed field of characteristic $\neq 2$. Let $\mathbb{D}$ be the "super skew field" $\mathbb{D} := k[\theta]$, $\theta$ odd, $\theta^2 = -1$. $\mathbb{D}$ is a noncommutative, associative superalgebra. Any homogeneous nonzero element of $\mathbb{D}$ is invertible. As we shall see, $\mathbb{D}$ is an example of a {\it central simple superalgebra}, which notion we now define.

Given any $k$-superalgebra $A$, we can define a superalgebra homomorphism $\psi: A \otimes_k A^o \to
\underline{End}_k(A)$ by:

\begin{equation*}
a \otimes b \mapsto (x \mapsto (-1)^{|b||x|}axb)
\end{equation*}\

\noindent Here $A^o$ denotes the opposite superalgebra of $A$. We say that a $k$-superalgebra $A$ is {\it central} if its supercenter equals $k$, and $A$ is {\it simple} if $A$ has no non-trivial two-sided homogeneous ideals.

The super Artin-Wedderburn theorem (cf. \cite{Var}) then states that:

\begin{thm}
Let $A$ be a superalgebra over a field $k$, finite dimensional as a $k$-super vector space. Then $A$ is central simple over $k$ if and only if $\psi: A \otimes A^o \to \underline{End}_k(A)$ is an isomorphism of $k$-superalgebras.
\end{thm}

\noindent We emphasize that the $\underline{End}$ appearing in the statement of the theorem is the ``internal $\underline{End}$" (i.e. the superalgebra of ungraded endomorphisms), not the categorical $End$ (i.e., the even subalgebra of even endomorphisms).

In \cite{DM}, it is shown that the super skew field $\mathbb{D}/k$ is a central simple superalgebra, and that it generates the super Brauer group $sBr$ of $k$, of Brauer equivalence classes of central simple superalgebras over $k$.

In ungraded commutative algebra, the notion of central simple algebra over a field $k$ is generalized to the category of algebras over a commutative ring by adopting the conclusion of the Artin-Wedderburn theorem as a definition. The resulting objects are called {\it Azumaya algebras}. We define the super analogue as follows:

\begin{defn}
Let $A$ be a superalgebra over a supercommutative ring $R$. $A$ is a {\it super Azumaya algebra} over $R$ if $A$ is a faithful, finitely generated projective $R$-module, and the natural homomorphism $\psi: A \otimes_R A^o \to \underline{End}_R(A)$ is an isomorphism of $R$-superalgebras.
\end{defn}


We then have the following:

\begin{prop}\label{prop:superAzumayaalg}
Let $k$ be an algebraically closed field, $char(k) \neq 2$ and $R$ be a commutative $k$-superalgebra. Then $\mathbb{D}_R :=\mathbb{D} \otimes_k R$ is a super Azumaya algebra over $R$.
\end{prop}

\begin{proof}
Since $\mathbb{D}_R$ is a free $R$-module, it is certainly faithful, finitely-generated, and projective. $\{1 | \theta\}$ is a homogeneous $R$-basis of $\mathbb{D}_R$, and so $\{1 \otimes 1, \theta \otimes \theta \, | \, 1 \otimes \theta, \theta \otimes 1\}$ is a homogeneous $R$-basis of $\mathbb{D} \otimes_R \mathbb{D}^o$. Using the basis ${1, \theta}$ of $\mathbb{D}$ to identify $\underline{End}_R(\mathbb{D}_R)$ with the matrix superalgebra $M_{1|1}(R)$, one sees that:\\

\begin{align}
\begin{aligned}\label{eq: matrixbasis}
\psi(1 \otimes 1) &= \left( \begin{array}{c|c}
1 & 0\\
\hline
0 & 1
\end{array}
\right) \hspace{5mm}
\psi(\theta \otimes 1) = \left( \begin{array}{c|c}
0 & -1\\
\hline
1 & 0
\end{array}
\right)\\
\psi(1 \otimes \theta) &= \left( \begin{array}{c|c}
0 & 1\\
\hline
1 & 0
\end{array}
\right) \hspace{5mm}
\psi(\theta \otimes \theta) = \left( \begin{array}{c|c}
-1 & 0\\
\hline
0 & 1
\end{array}
\right).
\end{aligned}
\end{align}\\

\noindent It is readily checked that these matrices form an $R$-basis of $M_{1|1}(R)$. (This requires the fact that $2$ is invertible in $R$). Thus $\psi$ sends an $R$-basis of $\mathbb{D} \otimes \mathbb{D}^o$ to one of $M_{1|1, R}$, hence must be an isomorphism.
\end{proof}








At this point, one ought to generalize our construction to obtain sheaves of super Azumaya algebras, locally isomorphic to $\mathbb{D}_R$. However, we will work in a more restrictive category rather than pursuing this line of development. Let $B/k$ be a superscheme over an algebraically closed field $k$, $char(k) \neq 2$. We define the sheaf $\underline{\mathbb{D}}_B$ by:

\begin{align*}
\underline{\mathbb{D}}_B(U) := \mathbb{D}_k \otimes_k \mathcal{O}_B(U)
\end{align*}

\bigskip

The sheaf $\underline{\mathbb{D}}^o_B$ may be defined in a completely analogous fashion. The following properties of $\underline{\mathbb{D}}_B$ follow from standard arguments and Prop. \ref{prop:superAzumayaalg}:\\

\begin{prop}\label{azumayasheaf}\
\begin{itemize}
\item $\underline{\mathbb{D}}_B$ is a trivial, rank $1|1$ locally free sheaf of $\mathcal{O}_B$-modules.\

\item $\underline{\mathbb{D}}_B$ is a sheaf of super Azumaya algebras over $\mathcal{O}_B$: for any point $b \in |B|$, there exists a Zariski open set $U \ni b$ such that $\phi: \underline{\mathbb{D}}_B(U) \otimes \underline{\mathbb{D}}^o_B(U) \to \underline{End}(\underline{\mathbb{D}}_B(U))$ is an isomorphism.
\end{itemize}
\end{prop}

In the proof of \ref{azumayasheaf}, one must use the fact that the structure of $k$-superscheme on $B$ implies that $2$ is invertible in $\mathcal{O}_B$. One can prove an analogue of Prop. \ref{azumayasheaf} for $\underline{\mathbb{D}}^o_B$ in the same way.

More generally, we will consider the category of relative superschemes $X/B$ for $B$ a Noetherian $k$-superscheme. We define a sheaf of super Azumaya algebras ${\mathbb{D}}_{X/B}$ on $\pi:X \to B$ by:

\begin{align*}
\underline{\mathbb{D}}_{X/B}\ := \pi^*(\underline{\mathbb{D}}_B)
\end{align*}

\bigskip

The relative version of Prop. \ref{azumayasheaf} holds in the category of $B$-superschemes:

\begin{prop}\label{prop:relativeazumayasheaf}\
\begin{itemize}
\item $\underline{\mathbb{D}}_{X/B}$ is a trivial rank $1|1$ locally free sheaf of $\mathcal{O}_X$-modules.\

\item $\underline{\mathbb{D}}_{X/B}$ is a sheaf of super Azumaya algebras over $\mathcal{O}_X$: for any point $x \in |X|$, there exists a Zariski open set $U \ni x$ such that $\phi: \underline{\mathbb{D}}_{X/B}(U) \otimes \underline{\mathbb{D}}^o_{X/B}(U) \to \underline{End}(\underline{\mathbb{D}}_{X/B}(U))$ is an isomorphism.
\end{itemize}
\end{prop}

\begin{proof}
The first statement follows from the fact that the inverse image of a trivial locally free sheaf of $\mathcal{O}_B$-modules is also locally free over $\mathcal{O}_X$ of the same rank and trivial. 

For the second, note that for any $x$ there is an open set $U \subseteq X$, $x \in U$, such that $f(U)$ is contained in an open set $V$ in $B$ and $\mathbb{D}_B(V)$ is a super Azumaya algebra with $\mathcal{O}_B(V)$-basis $1, \theta$. Then the proof of the claim is completely analogous to Prop. \ref{prop:superAzumayaalg}, replacing $k$ with $\mathcal{O}_B(V)$ and $R$ with $\mathcal{O}_X(U)$. The only point to note is that the direct limit of $\underline{\mathbb{D}}_B(W)$ over all open $W$ containing $f(U)$ is a super Azumaya algebra because $\underline{\mathbb{D}}_B(W)$ is super Azumaya for any open $W \subseteq V$.

\end{proof}

In the future, when working in the relative category of $X/B$, we will occasionally abuse notation and use $\underline{\mathbb{D}}_X$ to denote $\mathbb{D}_{X/B}$.\\

\noindent {\it Remark.} 1). It would be interesting to extend this theory to those cases where $2$ is not invertible.\\

\noindent 2). $\underline{\mathbb{D}}_B$ is not the only possible sheaf of super Azumaya algebras on $B$ which is locally isomorphic to $\mathbb{D} \otimes_k \mathcal{O}_B$; for instance, one could tensor $\mathbb{D}_B$ with any locally free sheaf of rank $1|0$ on $B$.

Here the structure morphism $B \to Spec(k)$ allows us to pull back $\mathbb{D}_k$ to $B$ in a canonical fashion, giving us a natural sheaf of super Azumaya algebras $\underline{\mathbb{D}}_B$ on $B$, and the structure morphism $f:X \to B$ then gives a canonical pullback of $\mathbb{D}_B$ to $X/B$.

$\mathbb{D}_k$ is canonical as well, in the following sense: the Brauer equivalence class of $\mathbb{D}_k$ is the generator of the super Brauer group $sBr(k)$ of $k$, which is isomorphic to $\mathbb{Z}_2$, (see section of \cite{DM}). Thus $\mathbb{D}_{X/B}$ is the most ``natural" way of extending $\mathbb{D}_k$ to a sheaf of super Azumaya algebras on $X/B$, in the sense that it involves no arbitrary choices, only the given structure morphisms.

Presumably a more complete understanding of the situation would entail developing the theory of the super Brauer group of a superscheme along the lines of \cite{Gro}. We speculate that those Brauer equivalence classes in the super Brauer group of $B$ which are represented by sheaves of super Azumaya algebras locally isomorphic to $\underline{\mathbb{D}}_B$ each correspond to fundamentally different ``twisted" versions of $\Pi$-projective geometries over $B$. What we treat in this paper might justifiably be called the ``untwisted" $\Pi$-projective geometry.\\

\subsection{$\mathbb{D}$-modules.}

In this section we shall discuss the theory of $\mathbb{D}_R$-modules. Since $\mathbb{D}$ is noncommutative, we must take care to maintain the distinction between left and right $\mathbb{D}$-modules. Often we denote $\mathbb{D}_R$ by $\mathbb{D}$ to save notation; the omission of the base ring should not cause any confusion.

\begin{defn}
A {\it left } (resp. {\it right) $\mathbb{D}_R$-module} is an $R$-module $M$ with an $R$-algebra homomorphism (resp. antihomomorphism) $\mathbb{D} \to \underline{End}_R(M)$. 
\end{defn}

It is readily seen that a left $\mathbb{D}$-action on $M$ is completely equivalent to the choice of an odd $R$-endomorphism $\phi$ of $M$ such that $\phi^2 = -1$. Namely, suppose given a left action of $\mathbb{D}$ on $M$; then the left action of $\theta$ on $M$ is an odd $R$-endomorphism whose square is $-1$. Conversely, given an odd $R$-endomorphism $\phi$ such that $\phi^2 = -1$, a left action of $\mathbb{D}$ on $M$ is given by:

\begin{equation*}
(a + b\theta) \cdot v = av + b\phi(v)
\end{equation*}\

In order to comply with our convention that endomorphisms of modules act on the left, we often convert a right $\mathbb{D}$-module into a left $\mathbb{D}^o$-module via the usual formula:

\begin{equation*}
s \cdot m := (-1)^{|s||m|} m \cdot s,
\end{equation*}\

where $\cdot$ denotes the $\mathbb{D}$-module action on the right and the $\mathbb{D}^o$-module action on the left.


From this point of view, a right $\mathbb{D}$-action on an $R$-module $M$ is equivalent to specifying an odd $R$-endomorphism $\phi'$ of $M$ such that $(\phi')^2 = 1$. 

\begin{defn}
A {\it homomorphism} $f: M \to N$ of left (right) $\mathbb{D}$-modules is a homomorphism of $R$-modules that intertwines the actions of $\mathbb{D}$ on $M$ and on $N$. $f$ is a {\it morphism} if $f$ preserves parity.
\end{defn}

Now that we have a notion of morphism, we have categories ${}_\mathbb{D} \mathfrak{M}$ (resp. $\mathfrak{M}_\mathbb{D}$ of left (resp. right) $\mathbb{D}$-modules, as well as the category ${}_\mathbb{D} \mathfrak{M}_\mathbb{D}$ of $\mathbb{D}$-bimodules. We can also define the categorical $Hom$ (i.e., parity-preserving homomorphisms) and internal $\underline{Hom}$ (all homomorphisms) in these categories, as one can for module categories over any associative super ring.

In particular, for a left (right) $\mathbb{D}$-module $M$, the $\mathbb{D}$-dual $M^\vee := \underline{Hom}_\mathbb{D}(M, \mathbb{D})$ is well-defined. $M^\vee$ is a left (right) $\mathbb{D}$-module in the usual way, by (left) right multiplication in $\mathbb{D}$.

One can also define a {\it free} $\mathbb{D}$-module on an {\it ungraded} basis set $I$, via the usual universal property. If $I$ is a finite set, the {\it rank} of the free $\mathbb{D}$-module on $I$ is defined to be $|I|$. That the rank is well-defined follows from the fact that a free $\mathbb{D}$-module of rank $n$ is also a free $R$-module of rank $n|n$ and that a supercommutative ring $R$ satisfies the invariant basis number property.

We note that the superrank of a free $\mathbb{D}$-module is not a well-defined notion: for instance, $\mathbb{D}$, regarded as an (e.g. left) $\mathbb{D}$-module, has even basis $\{1\}$ or odd basis $\{\theta\}$. This is a consequence of the noncommutativity of $\mathbb{D}$.

Owing to this noncommutativity, the theory of $\mathbb{D}$-modules is quite involved. However, we have the following extremely important special case, which later serves as a model for $\Pi$-invertible sheaves:

\begin{prop}\label{prop:canonicalbasis}
Let $k$ be an algebraically closed field, $char(k) \neq 2$, $R$ a commutative $k$-superalgebra, and $M$ a right $\mathbb{D}_R$-module, free of rank $1|1$ over $R$. Then $M \cong \mathbb{D}_R$ in $\mathfrak{M}_{\mathbb{D}_R}$.
\end{prop}

\begin{proof}
By the previous discussion, $M$ is an $R$-module with an odd endomorphism $\phi$, $\phi^2 = 1$. Choose a basis for $M$ as an $R$-module. It is easily seen that $\phi^2 = 1$ if and only if the matrix representing $\phi$ in this basis has the form

\begin{equation*}
P :=  \left(\begin{array}{c|c}
\alpha & a \\
\hline
a^{-1} & -\alpha \\
\end{array}\right)
\end{equation*}

\bigskip

\noindent in this basis, with $a, \alpha \in R$. We conjugate $P$ by the invertible matrix:

\begin{equation*}
B :=  \left(\begin{array}{c|c}
a^{-1} & -\alpha \\
\hline
0 & 1 \\
\end{array}\right)
\end{equation*}

\bigskip

\noindent obtaining:

\begin{align*}
P' := BPB^{-1} &=
 \left(\begin{array}{c|c}
a^{-1} & -\alpha \\
\hline
0 & 1 \\
\end{array}\right)
 \left(\begin{array}{c|c}
\alpha & a \\
\hline
a^{-1} & -\alpha \\
\end{array}\right)
 \left(\begin{array}{c|c}
a & a \alpha \\
\hline
0 & 1\\
\end{array}\right)\\
&= \left(\begin{array}{c|c}
0 & 1 \\
\hline
1 & 0 \\
\end{array}\right)
\end{align*}

\bigskip

Applying the change of basis matrix $B$ to our original basis, we obtain a basis $\{e | f \}$ such that $\phi(e) = f, \phi(f) = e$, hence a right $\mathbb{D}$-module isomorphism $\mathbb{D} \to M$.

\end{proof}

\noindent {\it Remark.} As previously, one can prove a completely analogous proposition for a left $\mathbb{D}$-
module, free of rank $1|1$ over $R$.

\subsection{\bf $\mathbb{D}^*$ and the group superscheme $\mathbb{G}^{1|1}_m$.}

For now, let us work in the category of $B$-superschemes, where $B$ is an arbitrary superscheme. Following \cite{ma1}, we define the group superscheme $\mathbb{G}^{1|1}_m$ over $B$, whose functor of points is given by:

\begin{align*}
T \mapsto [\Gamma(\mathcal{O}_T)]^*
\end{align*}

\medskip

\noindent for any $B$-superscheme $T$. Hence $\mathbb{G}^{1|1}_m(T)$ consists of all global sections $a + \alpha$ of $\mathcal{O}_T$, where $a$ is even and invertible, $\alpha$ odd. Similarly, we may define a sheaf of groups $\underline{\mathbb{G}^{1|1}_m}$, whose sections on an open set $U \subset X$ are given by:

\begin{align*}
\underline{\mathbb{G}^{1|1}_m}(U) =  [\mathcal{O}_X(U)]^*
\end{align*} 

\medskip

\noindent This is evidently just the sheaf of groups $\mathcal{O}^*_X$.

The following proposition (for the analytic category) is from \cite{ma1}:

\begin{prop}
The functor $\mathbb{G}^{1|1}_m: (Superschemes/B) \to (Groups)$ is represented by
the affine $B$-superscheme $\mathbb{A}^{1|1}_B \backslash \{0\}$, with group law given in terms of the functor of points by:

\begin{align*}
(a, \alpha) \cdot (a', \alpha') := (aa' + \alpha \alpha', a \alpha' + a' \alpha)
\end{align*}
\end{prop}

\begin{proof}
It is well known that the functor of points of $\mathbb{A}^{1|1}_B \backslash \{0\}$ is

\begin{align*}
\underline{\mathbb{A}^{1|1}_B \backslash \{0\}}(T) = \{(a, \alpha) : a, \alpha \in \Gamma(\mathcal{O}_T), \; a \text{ even and invertible, } \alpha \text{ odd} \}.
\end{align*}

\medskip

\noindent for $T$ a $B$-superscheme. $(a, \alpha) \mapsto a + \alpha$ is the desired isomorphism between the functor of points of $\mathbb{A}^{1|1}_B \backslash \{0\}$ and the functor $\mathbb{G}^{1|1}_m$. One checks readily that this isomorphism preserves the group laws.
\end{proof}

Occasionally it will be convenient to imbed $\mathbb{G}^{1|1}_m$ into $SL(1|1, \mathcal{O}_B)$ as a closed subsupergroup via:

\begin{align*}
a + \alpha \mapsto \left( \begin{array}{c|c}
a & \alpha \\
\hline
\alpha & a
\end{array}
\right).
\end{align*}

\bigskip

\noindent It is straightforward to check that the Berezinian of an element of this subsupergroup is $1$. This imbedding is also valid for the sheaf $\underline{\mathbb{G}}^{1|1}_m$ into $\underline{SL}(1|1)$.

Now we make the further assumption that $B$ is a superscheme over an algebraically closed field $k$, $char(k) \neq 2$. Then the sheaf $\underline{\mathbb{D}}_B$ of super Azumaya algebras naturally gives rise to a group superscheme $\mathbb{D}^*_B$ over $B$, via the functor of points:

\begin{align*}
T \mapsto [\Gamma(T, f^*\underline{\mathbb{D}}_B)]^*_0
\end{align*}

\medskip

\noindent for any $B$-superscheme $f: T \to B$. When the base $B$ is understood, we will sometimes write $\mathbb{D}^*$ for $\mathbb{D}^*_B$. Similarly, for a fixed $B$-superscheme $f:X \to B$ we have a sheaf of groups $\underline{\mathbb{D}}^*_X$ on $X$, defined by:

\begin{align*}
\underline{\mathbb{D}}^*_X(U) := [f^*(\underline{\mathbb{D}}_B)(U)]^*_0.
\end{align*}

\medskip

Deligne has pointed out that there is a natural isomorphism $\mathbb{G}^{1|1}_m \to \mathbb{D}^*$, given on the level of $T$-points
by:

\begin{align*}
a + \alpha \mapsto a + \theta \alpha
\end{align*}

\medskip

\noindent where $a, \alpha \in \mathcal{O}_T$, $a$ even and invertible, $\alpha$ odd. We may identify the sheaf of groups $\underline{\mathbb{G}}^{1|1}_m = \mathcal{O}^*_X$ with $\underline{\mathbb{D}}^*_X$ in the same fashion.

\section{Construction of Manin's $\Pi$-projective space $\mathbb{P}_\Pi(V)$}

We work in the category of $B$-superschemes, $B$ a superscheme over an algebraically closed field $k$, $char(k) \neq 2$.

To motivate our construction, let us take $B = Spec(k)$, and consider the $\mathbb{D}$-bimodule $V = \mathbb{D}^n$. Then the action of $\mathbb{D}^*$ on $V$ by {\it left} scalar multiplication is an automorphism of the {\it right} $\mathbb{D}$-module $V_\mathbb{D}$. In particular, it preserves the rank $1$ right $\mathbb{D}$-submodules of $V$. Hence for any rank $1$ right $\mathbb{D}$-submodule $W \subset V$, the restriction of left scalar multiplication by $\mathbb{D}^*$ to $W$ makes sense and naturally identifies $\mathbb{D}^*$ with the (super)group $Aut_{\mathbb{D}}(W_{\mathbb{D}})$. Since any nonzero even element $v$ of $V$ generates a rank $1$ right $\mathbb{D}$-submodule of $V$ and $v, v'$ generate the same submodule iff $v' = gv$ for some $g \in \mathbb{D}^*$, the ``quotient" of $V \backslash \{0\}$ by the left action of $\mathbb{D}^*$ should be identified with the space of right $\mathbb{D}$-lines (i.e., rank $1$ $\mathbb{D}$-submodules of $V_\mathbb{D}$). We shall make this precise for any $B$ by first constructing the quotient $V \backslash \{0\} / \mathbb{D}^*$, then characterizing morphisms into this quotient. 

It would be very interesting to develop an analogue of the $Proj$ construction for the category of $\mathbb{D}$-modules in order to extend these results to the case of arbitrary $k$-superschemes $B$ and obtain an invariant way of producing relative $\Pi$-projective spaces over arbitrary $B$-superschemes. We plan to address this topic in future work.

Let $(E, \psi, \phi)$ be a locally free sheaf of $\underline{\mathbb{D}}_B$-bimodules on $B$ of rank $n+1$ (here $\psi$ yields the left $\mathbb{D}$-action, $\phi$ the right $\mathbb{D}$-action). 

$E$ is, in a natural way, a locally free sheaf of $\mathcal{O}_B$-modules of rank $n+1|n+1$. To this sheaf is naturally associated the family of relative affine superspaces $\mathbb{A}(E)_B$ over $B$, given by

\begin{equation*}
\mathbb{A}(E)_B = {\bf Spec}(\underline{Sym}_{\mathcal{O}_B}(E^*))
\end{equation*}

\medskip

\noindent Here ${\bf Spec}$ denotes global super $Spec$ over $B$. We will denote this affine $B$-superscheme by $\underline{E}$ to distinguish it from the sheaf $E$; the structure morphism will be denoted by $\rho: \underline{E} \to B$. It is well known that $\underline{E}$ represents the functor of points:

\begin{align*}
(Superschemes/B) &\to (Sets)\\
(f: T \to B) &\to \Gamma(T, f^*(E))_0.
\end{align*}

\medskip

We shall give a construction of the Manin $\Pi$-projective bundle $\mathbb{P}_\Pi(E, \psi, \phi)$ over $B$ as the quotient of $\underline{E} \backslash \{0\}$ by the action of the group superscheme $\mathbb{D}^*_B$.

The right $\underline{\mathbb{D}}_B$-action on $E$ by scalar multiplication is given by:

\begin{equation}\label{eq:actionformula}
v \cdot (a + \theta \alpha) := v a + (-1)^{|v|} \phi(v) \alpha
\end{equation}

\medskip

\noindent for a homogeneous element $a + \theta \alpha$ of $\underline{\mathbb{D}}_B$.

We now turn to the left action of $\mathbb{D}^*_B$ on the {\it superscheme} $\underline{E}$ induced by the left action of $\mathbb{D}_B$. It will be more convenient for us to convert it into a right action by the standard device of composing with the inversion antihomomorphism:

\begin{align}
v \cdot (t + \theta \tau) :=  (t + \theta \tau)^{-1} v
\end{align}

\medskip

\noindent \noindent where this equation is now interpreted in terms of the functor of points. This right $\mathbb{D}^*_B$-action induces a left $\mathbb{D}^*_B$-action on $\mathcal{O}_{\underline{E}}$. We emphasize that the right action of $\mathbb{D}^*_B$ so defined is completely distinct from the previously defined right action of $\underline{\mathbb{D}}_B$; indeed, the two actions commute.

The zero-section of $E$ embeds $B$ canonically into $\underline{E}$ as a closed subsuperscheme; hence the complement of the image of the zero-section $\underline{E} \backslash \{0\}$ is an open $B$-subsuperscheme of $\underline{E}$. The actions of $\mathbb{D}^*$ and of $\underline{\mathbb{D}}$ are linear, thus restrict to $\underline{E} \backslash \{0\}$.

We have the reduction morphisms $\mathbb{G}_m \to \mathbb{D}^*$, $\underline{E}\backslash\{0\}
\to \overline{\underline{E}}\backslash\{0\})$, which are the identity on the underlying topological spaces. Let $\pi: |\underline{\overline{E}}\backslash\{0\}| \to |\mathbb{P}_{\overline{B}}(\overline{E})|$ be the map on the underlying spaces induced by the quotient morphism $\underline{\overline{E}} \backslash \{0\} \to \mathbb{P}_{\overline{B}}(\overline{V})$. Recall that $|\underline{V}\backslash\{0\}| = |\underline{\overline{E}}\backslash\{0\}|$.

We will define $\mathbb{P}_{\Pi, B}(E)$ as follows. Its underlying topological space will be the underlying space $|\mathbb{P}_{\overline{B}}(\overline{E})|$ of the $\overline{B}$-projective space, $\mathbb{P}_{\overline{B}}(\overline{E})$. We will construct a sheaf of $\mathcal{O}_B$-superalgebras $\mathcal{O}_{\mathbb{P}_{\Pi, B}(E)}$ on $|\mathbb{P}_{\overline{B}}(\overline{E})|$ such that $(\mathbb{P}(E_{red}), \mathcal{O}_{\mathbb{P}_{\Pi, B}(E)})$ is a $B$-superscheme. 

Let $U \subseteq \mathbb{P}_{\overline{B}}(\overline{E})$ be any Zariski open set. Then $U' :=
\pi^{-1}(U)$ is an open subset of $|\overline{\underline{E}}\backslash\{0\}| = |\underline{E}\backslash\{0\}| $. The open subscheme $(U', \mathcal{O}_{\overline{\underline{E}} \backslash \{0\}}|_{U'})$ of $\overline{\underline{E}} \backslash \{0\}$ is well-known to be $\mathbb{G}_m$-invariant. This implies that $\mathbb{D}^*_B$ acts on the $B$-superscheme $\underline{U}':= (U', \mathcal{O}_{\underline{E}\backslash\{0\}}|_{U'}) \subseteq \underline{E} \backslash\{0\}$, since the restriction $a|_{\mathbb{D}^* \times_B \underline{U}'}$ to $\mathbb{D}^* \times_B \underline{U}'$ of the action morphism $a: \mathbb{D}^* \times_B \underline{E} \backslash \{0\} \to \underline{E} \backslash \{0\}$ maps into $U'$.

\begin{defn}
Let $X$ be a $B$-superscheme, $G$ a group superscheme over $B$, and $a \text{ (resp. } p_2): G \times_B X \to X$ an action of $G$ on $X$ (resp. the projection on the second factor). A function $f \in \mathcal{O}_X$ is {\it $G$-invariant} if and only if $a^*(f) = p_2^*(f)$.
\end{defn}

\medskip

We may now define the sheaf $\mathcal{O}_{\mathbb{P}_\Pi(E)}$ by:

\begin{align}
\mathcal{O}_{\mathbb{P}_\Pi(E)} (U) := \mathcal{O}^{\mathbb{D}^*}_{\underline{E}\backslash\{0\}}(U')
\end{align}

\medskip

where $\mathcal{O}^{\mathbb{D}^*}_{\underline{E}\backslash\{0\}}(U')$ denotes the supercommutative ring of $\mathbb{D}^*$-invariant sections of $\mathcal{O}_{\underline{E}\backslash \{0\}}$ on $U'$. One checks that this assignment is indeed a sheaf of $\mathcal{O}_B$-modules on $|\mathbb{P}_{\overline{B}}(\overline{E})|$.

\begin{defn}
The {\it $\Pi$-projective superspace} over $B$ is the ringed superspace $(|\mathbb{P}_{\overline{B}}(\overline{E})|,\mathcal{O}_{\mathbb{P}_{\Pi, B}(E)})$ over $B$.
\end{defn}

\subsection{Affine cells of $\mathbb{P}_\Pi(E)$}

For simplicity, we will restrict ourselves in this section to the case where $B$ is an affine $k$-superscheme, $E$ a free sheaf of $\mathbb{D}$-modules on $B$. (The general case will be treated in later papers). We will prove that $\mathbb{P}_\Pi(E)$ has a Zariski open covering by $B$-superaffine spaces. This will imply in particular that $(\mathbb{P}_{\overline{B}}(\overline{E}), \mathcal{O}_{\mathbb{P}_{B, \Pi}(E)})$ is a smooth $B$-superscheme. 

For this purpose, we may work locally on $B$, in a trivializing affine cover for $E$ as a $\mathbb{D}_B$-bimodule. So we may assume that $B = Spec(A)$, for some affine $k$-superalgebra $A$, and that $\mathcal{O}_V(B)$ is a free $\mathbb{D}_B$-bimodule of rank $n+1$, some $n$. Hence there is an $\mathcal{O}_B$-basis $\{e_i | f_i\}$ of $V$ such that $\phi(e_i) = f_i, \phi(f_i) = e_i$, $\psi(e_i) = f_i, \psi(f_i) = -e_i$ for $i = 0, \dotsc, n$.

Let $\{z_i | \zeta_i\}$ be linear functionals on $V$ dual to the basis $\{e_i | f_i\}$; we may then consider them as linear functions on $\underline{V}$. Similarly, let $t, \tau$ be linear functions on $\mathbb{D}^*$ dual to $1, \theta$. Then the action of a $T$-point $t + \theta \tau$ of $\mathbb{D}^*_B$ on a $T$-point $\sum_i e_i z_i + f_i \zeta_i$ of $\underline{V}$ becomes:

\begin{equation}
\begin{aligned}
&\sum_i (e_i z_i +  f_i \zeta_i) \cdot (t + \theta \tau)\\
= \, & \sum_i  (e_i z_i  + f_i \zeta_i)t^{-1} -  \psi(e_i z_i + f_i \zeta_i) t^{-2} \tau\\
=\, & \sum_i e_i (t^{-1} z_i - t^{-2} \tau \zeta_i) + f_i (t^{-1} \zeta_i - t^{-2} \tau z_i)
\end{aligned}
\end{equation}

\medskip

In these expressions we are abusing notation and writing $z_i$ for the pullback of $z_i$ to $\Gamma(\mathcal{O}_T)$, etc. This equality holds good independent of the choice of $T$-point. Hence the right $\mathbb{D}^*$ action on $\underline{V}$ may be written in terms of the $z_i$ and $\zeta_i$ as:

\begin{equation}
\begin{aligned}\label{eqn:actionincoordinates}
&(z_0, \zeta_0, \dotsc, z_n, \zeta_n) \cdot (t, \tau) = (t^{-1} z_0  - t^{-2} \tau \zeta_0, t^{-1} \zeta_0 - t^{-2} \tau z_0, \dotsc, \\ 
& t^{-1} z_n - t^{-2} \tau \zeta_n, t^{-1} \zeta_n - t^{-2} \tau z_n)
\end{aligned}
\end{equation}

\medskip

\noindent {\it Remark.} Although we have chosen specific coordinates for $\underline{V}$ in which the $\underline{\mathbb{D}}$- and $\mathbb{D}^*$-actions take a particularly simple form in order to facilitate our calculations, these actions were defined purely in terms of the $\mathbb{D}$-bimodule structure of $V$. Hence our constructions will depend only on the $\mathbb{D}$-bimodule structure of $V$, and not on any arbitrary choices.\\

To this end, let us consider the open subset $U'_i := D(z_i)$ of $\underline{E} \backslash \{0\}$. The image of $U'_i$ in $\mathbb{P}_{\overline{B}}(\overline{E})$ is then the open subset $U_i =  \{[\overline{z}_0, \dotsc \overline{z}_n ] : \overline{z}_i \neq 0\}$. The $\{U_i\}$, $i = 0, \dotsc, n$, form a Zariski open cover of $|\mathbb{P}_{\overline{B}}(\overline{E})|$.

We may now characterize the rings $\mathcal{O}_{\underline{E}\backslash\{0\}}^{\mathbb{D}^*}(U'_i)$.

\begin{prop}
$\mathcal{O}_{\underline{E}\backslash\{0\}}^{\mathbb{D}^*}(U'_i)$ is the $A$-superalgebra generated over $A$ by the functions:

\begin{align*}
&w^j_i := \frac {z_j} {z_i} - \frac {\zeta_i \zeta_j} {z^2_i}\\
&\eta^j_i := \frac {\zeta_j} {z_i} - \frac {z_j \zeta_i} {z^2_i}
\end{align*}\\

\noindent where $j \in \{0, 1, \cdots, \hat{i}, \cdots, n\} $. In
particular, $\mathcal{O}^{\mathbb{D}^*}(U_i)$ is a finitely-generated $A$-superalgebra.
\end{prop}

\begin{proof} The $\mathbb{D}^*$-invariance of the functions $w^j_i, \eta^j_i$ is shown by a direct calculation. It remains to be shown that $w^j_i$, $\eta^j_i$ actually generate $\mathcal{O}^{\mathbb{D}^*} (U_i)$. For this we require the following lemma.

\begin{lem}\label{lem:zeroinvariant}
Let $s \in \mathcal{O}^{\mathbb{D}^*}(U_i)$ be a $\mathbb{Z}_2$-homogeneous invariant section. Suppose that $s$ is a multiple of $\zeta_i$. Then $s$ is identically zero.
\end{lem}

\begin{proof}
We begin by noting that $\pi^{-1}(U_i)$ is the affine $B$-superscheme with coordinate superalgebra $\mathcal{O}_{\pi^{-1}(U_i)} = A[z_0, \dotsc, z_n, \zeta_0, \dotsc, \zeta_n][z^{-1}_i]$. 

Since $\mathbb{D}^*$ and $\pi^{-1}(U_i)$ are both affine $B$-superschemes, we may work with their superalgebras of global functions. $s$ is in particular invariant under the subsupergroup $\mathbb{G}^{1|0}_m \subset \mathbb{D}^*$, which is true if and only if $s$ is a sum of rational functions of the form:

\begin{equation*}
s = \sum_{J,K} a_{JK}\frac {z^{p_1}_{j_1} z^{p_2}_{j_2} \cdots z^{p_{|J|}}_{j_{|J|}} \zeta_{k_1} \zeta_{k_2} \cdots \zeta_{k_{|K|}}} {z_i^{|K|+\sum_J p_j}}
\end{equation*}

\medskip

\noindent where $z_i$ does not appear in the numerator of any term, and $\zeta_i$ appears in the numerator of each term. Here $J, K$ are multiindices, and $a_{JK} \in \mathcal{O}_B$.

The equations that follow will all hold in $\mathcal{O}_{\mathbb{D}^* \times_B \pi^{-1}(U_i)} = \mathcal{O}_{\mathbb{D}^*} \otimes_B \mathcal{O}_{\pi^{-1}(U_i)} = A[z_0, \dotsc, z_n, \zeta_0, \dotsc, \zeta_n, t, \tau][z^{-1}_i]$. The pullback of $s$ by $(t + \theta \tau)^{-1}$ is:

\begin{align*}
&((t + \theta \tau)^{-1})^*(s)\\
=& \sum_{J,K} a_{JK} \frac {(t z_{j_1} + \tau \zeta_{j_1})^{p_1} \cdots (t z_{j_{|J|}}  + \tau \zeta_{j_{|J|}})^{p_{{|J|}}} (t \zeta_{k_1} + \tau z_{k_1}) \cdots (t \zeta_{k_{|K|}} + \tau z_{k_{|K|}})} {(t z_i + \tau \zeta_i)^{|K|+\sum_J p_j}}\\
=& \sum_{J,K} a_{JK} \frac {(t z_{j_1} + \tau \zeta_{j_1})^{p_1} \cdots (t z_{j_{|J|}} + \tau \zeta_{j_{|J|}})^{p_{|J|}} (t \zeta_{k_1} + \tau z_{k_1}) \cdots (t \zeta_{k_{|K|}} + \tau z_{k_{|K|}})} {(z_i t)^{|K|+\sum_J p_j}}
\end{align*}

\bigskip

\noindent The last equation holds since multiplication by $\tau \zeta_i$ annihilates the numerator of $(t + \tau \theta)^*(s)$ (by expanding the numerator as a polynomial in the $z$s, the $\zeta$s, $t$ and $\tau$, one sees by the assumptions of the proposition that every term must contain either $\zeta_i$ or $\tau$.) 

Let $P$ be the sum of all terms of $((t + \theta \tau)^{-1})^*(s)$ that do not contain $\tau$, and $Q$ the sum of those that do contain $\tau$. One may check by direct calculation that $P = s$. Hence $(t + \tau \theta)^*(s) = s$ implies that $Q = 0$. Then $(z_i t)^{|K|+\sum_J p_j} Q$ is the zero polynomial in $\mathcal{O}_{\mathbb{D}^* \times_B U_i}$.

Let us consider the terms of $(z_i t)^{|K|+\sum_J p_j} Q$ that contain $z_i$, call the sum of all such terms $Q'$. $Q'$ is a polynomial, and since the numerator of $s$ does not contain $z_i$, each term in $Q'$ contains only a linear power of $z_i$. We see that each term of $Q'$ must be a multiple of $z_i \tau$, obtained by substituting $z_i \tau$ for $\zeta_i$ in a corresponding term of $s$:

\begin{align*}
a_{j_1 \dotsc j_{|J|}, k_1 \dotsc k_{|K|}} z_{j_1}^{p_1} \dotsc z_{j_{|J|}}^{p_{|J|}} \zeta_{k_1} \dotsc \hat{\zeta_i} (z_i \tau) \dotsc \zeta_{k_{|K|}} t^{|K| + \sum_J p_j -1} 
\end{align*}

\bigskip

\noindent Conversely, every term of $s$ gives rise to a unique term of $Q'$ in this way. 

Since $z_i$ is algebraically independent from the other $z_j$s and $\zeta_k$s, $(z_i t)^{|K|+\sum_J p_j} Q = 0$ implies that $Q'$ must also be identically zero. (Alternatively, to see this one could differentiate the equation $(z_i t)^{|K|+\sum_J p_j} Q \\ = 0$ with respect to $z_i$). Hence all of the coefficients $a_{JK}$ must be zero, so that $s$ is identically zero.

\end{proof}







Now we show that any invariant section $s$ may be written as a
polynomial in the functions $w_j, \eta_j$. Suppose then that $s$ is
such a section. Note that any product

\begin{equation*}
\prod_{j \in J} w_j \prod_{k \in K} \eta_k = \prod_{j \in J} \left(\frac {z_j} {z_i} - \frac {\zeta_i \zeta_j} {z^2_i}\right) \prod_{k \in K} \left( \frac {\zeta_k} {z_i} - \frac {\zeta_i z_k} {z^2_i} \right)
\end{equation*}

\medskip

\noindent of the $w_j$ and the $\eta_k$ contains exactly one term that does not contain $\zeta_i$ in the numerator, namely, the rational function:

\begin{equation*}
z_{J} \zeta_K := \frac {z_{j_1} z_{j_2} \cdots z_{j_{|J|}} \zeta_{k_1} \zeta_{k_2} \cdots \zeta_{k_{|K|}}} {z_i^{|K|+\sum_J p_j}}
\end{equation*}

\medskip

\noindent Note that $z_i$ also does not appear in the numerator of this rational function. We shall refer to rational functions of this type as ``head terms.'' Conversely, note that given any pair of multiindices $J, K$ for which $i \notin J$ and $i \notin K$, we may produce an invariant section with head term $z_J \zeta_K$ by taking $\prod_{j \in J} w_j \prod_{k \in K} \eta_k$.

Let $\sum_{J, K} a_{JK} {z_J  \zeta_K}$ be the sum of all head terms
in $s$. Then the section:

\begin{equation*}
s' := s - \sum_{J, K} a_{JK} \prod_{j \in J} w_j \prod_{k \in K} \eta_k
\end{equation*}\\

\noindent is $\mathbb{D}^*$-invariant, being a difference of $\mathbb{D}^*$-invariant sections. By our remark about products of the $w_j$ and $\eta_k$, all terms of $s'$ must contain $\zeta_i$, since the only head terms of $a_{JK} \prod_{j \in J} w_j \prod_{k \in K} \eta_k$ are $a_{JK} z_J \zeta_K$, and these cancel with the corresponding head terms in $s$ by construction. Therefore $s'$ is identically zero by Lemma \ref{lem:zeroinvariant}, i.e.

\begin{equation*}
s = \sum_{J, K} a_{JK} \prod_{j \in J} w_j \prod_{k \in K} \eta_k
\end{equation*}

\medskip

\noindent which is what we wished to prove.

\end{proof}

As a consequence, we may now show that $\mathbb{P}_{\Pi, B}(V)$ so defined is actually covered by affine superspaces:

\begin{cor}
$\mathcal{O}_{\mathbb{P}^n_\Pi}(U_i)$ is a free commutative $A$-superalgebra on $n|n$ variables.
\end{cor}

\begin{proof}
Without loss of generality we may take $i = 0$, the argument being the same for the other values of $i$, after reindexing the variables. Let $C$ be the free $A$-superalgebra $A[y_1, \dotsc y_n] \otimes \Lambda[\tau_1, \dotsc, \tau_n]$ on $n|n$ variables. We define a homomorphism $F: C \rightarrow \mathcal{O}_{\mathbb{P}^n_\Pi} (U_i)$ by sending $y_j \mapsto w_j, \tau_j \mapsto \eta_j$. By the above proposition, $F$ is surjective.

We proceed to show that $F$ is injective as well. To this end, let
$P = \sum_{J,K} a_{JK} y_J \tau_K$. Here, as before, we will use the
multindex notation:

\begin{align*}
&y_J := y^{p_1}_{j_1}y^{p_2}_{j_2} \dotsc y^{p_{|J|}}_{j_{|J|}}\\
&\tau_K := \tau_{k_1} \tau_{k_2} \dotsc \tau_{k_{|K|}}
\end{align*}

\bigskip

\noindent We have:

\begin{align*}
F(P) &= \sum_{J,K} a_{JK} F(y_J) F(\tau_K)\\
&= \sum_{J,K} a_{JK} w_J \eta_K\\
&= \sum_{J,K} a_{JK}  \prod_{j \in J} \bigg( \frac {z_j} {z_0} - \frac {\zeta_0 \zeta_j} {z^2_0} \bigg) \prod_{k \in K} \bigg( \frac {\zeta_k} {z_0} - \frac {\zeta_0 z_k} {z^2_0} \bigg)\\
&=0
\end{align*}

\bigskip

For each pair of multi-indices $J,K$ there is a unique head term in
$F(P)$:\

\begin{equation*}
a_{JK} z_J \zeta_K := a_{JK} \prod_{j \in J} \frac {z_j} {z_0} \prod_{k \in K} \frac {\zeta_k} {z_0}
\end{equation*}

\bigskip

\noindent Since all other terms besides $a_{JK}z_J \zeta_K$ are multiples of $\zeta_i$, $F(P) = 0$ implies that $\sum_{JK} a_{JK} z_J \zeta_K =0$. But the rational functions $z_J \zeta_K$ are $\mathcal{O}_B$-linearly independent, so we conclude that $a_{JK} = 0$ for all multi-indices $J,K$, i.e. $P=0$, which is what we wanted to prove.
\end{proof}

From this, we deduce certain important properties of $\mathbb{P}_{\Pi, B}(V)$ from this result. First, this shows that $\mathcal{O}_{\mathbb{P}_{\Pi, B}(V)}$ is a sheaf of {\it local} super rings. Hence $\mathbb{P}_{\Pi, B}(V)$ is indeed a $B$-superscheme. Second, this implies that $\mathbb{P}_{\Pi, B}(V)$ is of finite type over $B$, and smooth over $B$.

\section{$\mathbb{P}^n_\Pi$ as a quotient}

In \cite{Levin}, it is stated without proof that $\mathbb{P}^n_\Pi$ is a quotient of $\mathbb{C}^{n+1|n+1}$ by $\mathbb{G}^{1|1}_m$. A more precise formulation of this statement is given by the following:

\begin{prop}
Let $B$ be an affine $k$-superscheme, $E$ a free $\mathbb{D}_B$-bimodule. Then $\underline{E} \backslash \{0\}$ is a $\mathbb{D}^*$-principal bundle over $\mathbb{P}_{\Pi, B}(E)$, via the projection map $\pi: \underline{E} \backslash \{0\} \to \mathbb{P}_{\Pi, B}(E)$. 
\end{prop}

\begin{proof}
We begin by noting that the open subsets $|U'_i| = \{(z_0, \dotsc z_n): z_i \neq 0\}$ of $|{\bf Spec}_B(E) \backslash \{0\}|$ are invariant under the action of the reduced group $\mathbb{G}_m$ of $\mathbb{D}^*$:

\begin{equation*}
t \cdot (z_0, \dotsc, z_n, \zeta_0, \dotsc \zeta_n) = (t^{-1} z_0, \dotsc, t^{-1} z_n, t^{-1} \zeta_0, \dotsc t^{-1} \zeta_n)
\end{equation*}

\noindent hence the $U'_i$ are invariant under the action of $\mathbb{D}^*$. We will show that the $U'_i$ are isomorphic, as $\mathbb{D}^*$-superschemes, to $\mathbb{D}^* \times_B U_i$, where the latter is regarded as a $\mathbb{D}^*$-superscheme by multiplication on the first factor.

We shall construct such an isomorphism $\rho: U'_i \to \mathbb{D}^* \times_B U_i$ with the aid of the invariant sections $w^j_i, \eta^j_i$ of $\mathcal{O}(U'_i)$. 

Let $t, \tau$ be coordinates on $\mathbb{D}^*$ (i.e. linear functionals that generate $\Gamma(\mathbb{D}^*, \mathcal{O}_{\mathbb{D}^*}$ as a sheaf of $\mathcal{O}_B$-superalgebras), and $z_0, \dotsc, z_n, \zeta_0, \dotsc \zeta_n$ the linear coordinates on ${\bf Spec}_B(V) \backslash \{0\}$. We define a $B$-morphism $\rho: U'_i \to \mathbb{D}^* \times_B U_i$ by:

\begin{align*}
&\Phi(z_0, \dotsc, z_n, \zeta_0, \dotsc, \zeta_n)\\
=&((z_i, \zeta_i ), \tfrac {z_0} {z_i} - \tfrac {\zeta_i \zeta_0} {z^2_i}, \dotsc, \tfrac {z_{i-1}} {z_i} - \tfrac {\zeta_i \zeta_{i-1}} {z^2_i}, \tfrac {z_{i+1}} {z_i} - \tfrac {\zeta_i \zeta_{i+1}} {z^2_i}, \dotsc,  \tfrac {z_n} {z_i} - \tfrac {\zeta_i \zeta_n} {z^2_i},\\
& \tfrac {\zeta_0} {z_i} - \tfrac {\zeta_i z_0} {z^2_i}, \dotsc, \tfrac {\zeta_{i-1}} {z_i} - \tfrac {\zeta_i z_{i-1}} {z^2_i}, \tfrac {\zeta_{i-1}} {z_i} - \tfrac {\zeta_i z_{i-1}} {z^2_i}, \dotsc, \tfrac {\zeta_n} {z_i} - \tfrac {\zeta_i z_n} {z^2_i})
\end{align*}

\medskip

Define $\Psi: \mathbb{D}^* \times_B U_i \to U'_i$ by:

\begin{align*}
&\Psi((t, \tau), w_0, \dotsc, w_{i-1}, w_{i+1}, \dotsc, w_n, \eta_0, \dotsc, \eta_{i-1}, \eta_{i+1}, \dotsc, \eta_n)\\
=&(sw_0 + \sigma \eta_0, \dotsc, sw_{i-1} + \sigma \eta_{i-1}, s, sw_{i+1} + \sigma \eta_{i+1},  \dotsc, sw_n + \sigma \eta_n, \\
& s\eta_0 + \sigma w_0, \dotsc, s\eta_{i-1} + \sigma w_{i-1}, \sigma, s\eta_{i+1} + \sigma w_{i+1}, \dotsc, s\eta_n + \sigma w_n)
\end{align*}

\medskip

Since the sections $w^j_i, \eta^j_i$ freely generate $\mathcal{O}_{\mathbb{P}^n_\Pi}$ as an $\mathcal{O}_B$-superalgebra on the open set $U_i = \pi(U'_i) \subset \mathbb{P}^n_\Pi$, the above equations do indeed define morphisms of $B$-superschemes.

A direct calculation, which is lengthy but straightforward and thus omitted, shows that $\rho$ and $\Psi$ are mutually inverse, so that $\rho$ is an isomorphism of $B$-superschemes. $\mathbb{D}^*$-equivariance of $\rho$ is checked similarly.

\end{proof}

\section{$\Pi$-invertible sheaves}

Let $X/B$ be a $B$-supermanifold ($B$-superscheme), where $B$ is a complex analytic supermanifold, or a superscheme. The following definition is due to Skornyakov \cite{ma1}:

\medskip

\begin{defn}
A {\it $\Pi$-invertible sheaf} on $X/B$ is a pair $(S, \phi)$, where $S$ is a locally free sheaf of $\mathcal{O}_{X/B}$-modules of rank $1|1$, and $\phi \in H^0(X, \underline{End}(S))$ is an odd endomorphism of $S$ such that $\phi^2 = 1$. A {\it morphism} of right $\Pi$-invertible sheaves $f: (S, \phi) \to (S',\phi')$ is a homomorphism of locally free sheaves $f: S \to S'$ such that $f \circ \phi = \phi' \circ f$.
\end{defn}

\medskip

In the category of $B$-superschemes, where $B$ a $k$-superscheme, $k$ an algebraically closed field of characteristic $\neq 2$, the concept of $\Pi$-invertible sheaf may be given a new interpretation, as suggested by Deligne \cite{Del1} in the complex analytic case: it is completely equivalent to the concept of a rank 1 locally free sheaf of right $\underline{\mathbb{D}}_{X/B}$ modules, and a morphism of $\Pi$-invertible sheaves is precisely the same thing as a morphism of right $\underline{\mathbb{D}}_{X/B}$-modules.

To show one direction, suppose $(S, \phi)$ is a $\Pi$-invertible sheaf on $X \to B$. The right $\underline{\mathbb{D}}_X$-action on $S$ is recovered by the formula: 

\begin{align*}
s \cdot (a + \theta \alpha) := s a + (-1)^{|s|} \phi(s) \alpha 
\end{align*}

\medskip

\noindent for any open set $U \subseteq X$ such that $\pi(U) \subseteq V$, where $V$ is an open subset of $B$ on which $\underline{\mathbb{D}}_B$ has basis $1, \theta$, $s \in \mathcal{O}_S(U)$, $a + \theta \alpha \in \mathcal{O}_{\underline{\mathbb{D}}_{X/B}}(U)$. Since such $U$ form a basis for the topology of $X$, this defines a right $\underline{\mathbb{D}}_X$-action on $S$. Then by Prop. \ref{prop:canonicalbasis}, $S$ is locally free of rank $1$ as a sheaf of right $\underline{\mathbb{D}}_X$-modules. 


For the converse, suppose that $S$ is a sheaf of locally free, rank 1 right $\underline{\mathbb{D}}_X$-modules. Then by Prop. \ref{prop:relativeazumayasheaf}, $S$ is a locally free, rank $1|1$ sheaf of $\mathcal{O}_X$-modules. The action of $\theta$ defines an odd endomorphism $\phi$ of $S$:

\begin{align*}
\phi(s) := (-1)^{|s|} s \cdot \theta
\end{align*}

\medskip

\noindent for $s$ a homogeneous section of $S$ over any open set $U \subseteq X$. One readily sees that $\phi$ is well-defined and $\mathcal{O}_X$-linear, and that $\phi^2 = 1$, so $(S, \phi)$ is a $\Pi$-invertible sheaf. Now it is routine to check that a morphism of $\Pi$-invertible sheaves $f: (S, \phi) \to (S', \phi')$ is precisely the same thing as a morphism $f: S \to S'$ of right $\underline{\mathbb{D}}_X$-modules.

It follows from this discussion that the transition functions of a $\Pi$-invertible sheaf $(S, \phi)$ on $X$ may be reduced to $GL(1, \underline{\mathbb{D}}_X) = \underline{\mathbb{D}}_X^*$, and by standard arguments in the cohomology theory of sheaves of groups, it may be shown that the set of isomorphism classes of $\Pi$-invertible sheaves is in bijective correspondence with the sheaf cohomology set $H^1(X, \underline{\mathbb{D}}^*_X)$.

\section{$\mathbb{D}$-hyperplane bundle on $\mathbb{P}_\Pi(E)$}

The $\Pi$-projective superspace $\mathbb{P}_\Pi(E)$ is endowed with a natural $\Pi$-invertible sheaf $\mathcal{O}_\Pi(1)$, analogous to the hyperplane bundle $\mathcal{O}(1)$ on ordinary projective space. Intuitively, the fiber of this $\Pi$-invertible sheaf over a point $W \in \mathbb{P}_\Pi(E)$ (i.e., a free, rank $1$ right $\mathbb{D}$-module of $E$) is the free, rank $1$ right $\mathbb{D}$-module $W^\vee$. 

In this section, we shall give a definition of $\mathcal{O}_\Pi(1)$ using the super skew field $\mathbb{D}$, describe its basic properties, and use it to characterize $B$-morphisms $X \to \mathbb{P}_{\Pi, B}(E)$ for any affine $B$-superscheme $X/B$. The existence and key properties of $\mathcal{O}_\Pi(1)$ were also mentioned in \cite{ma1}, without proofs.


\begin{defn}
Let $E$ be a locally free, rank $n$ sheaf of $\underline{\mathbb{D}}_B$-bimodules. The $\Pi$-invertible sheaf $\mathcal{O}_\Pi(1)$ is the sheaf defined by:

\begin{equation*}
\mathcal{O}_\Pi(1)(U) := p^*(E^\vee)(U).
\end{equation*}

\medskip

\noindent Here $E^\vee$ denotes the sheaf $\underline{Hom}(E_{\underline{\mathbb{D}}}, \underline{\mathbb{D}})$, $p: \mathbb{P}_\Pi(E) \to B$ the structure morphism.
\end{defn}

\noindent 

The first order of business is to verify that $\mathcal{O}_\Pi(1)$ so defined is indeed a $\Pi$-invertible sheaf.

\begin{prop}
$\mathcal{O}_\Pi(1)$ is a $\Pi$-invertible sheaf on $\mathbb{P}_{\Pi,B}(E)$.
\end{prop}

\begin{proof}
$\mathcal{O}_\Pi(1)$ inherits a natural right $\mathbb{D}$-module structure, given by the $\mathbb{D}$-action on $E^\vee$ by {\it right} multiplication.

To check local freeness, we may work locally on $B$, in an affine cover trivializing $E$ as a sheaf of $\mathbb{D}$-bimodules. So let us assume that $B = Spec(A)$, and that there is a $B$-basis $\{e_i | f_i\}$ of $\Gamma(E)$ such that $\phi(e_i) = f_i, \phi(f_i) = e_i, \psi(e_i) = f_i, \psi(f_i) = -e_i$. Let $\{z_i, \zeta_i\}$ be a basis of $B$-linear functionals on $E$, dual to $\{e_i, f_i\}$ respectively. 

We sketch the calculation that $s_j := z_j + \theta \zeta_j, \sigma_j := \zeta_j + \theta z_j$, $j = 0, \dotsc n$ is a $B$-basis of $V^\vee$. Suppose $s \in E^\vee$; we may as well assume $s$ is even. Since $s$ is $\mathbb{D}$-linear it must in particular be $\mathcal{O}_B$-linear. Then $s = \sum_j z_j a_j + \zeta_j \alpha_j + \theta (\zeta_j b_j + z_j \beta_j)$. (Right) $\mathbb{D}$-linearity of $s$ is equivalent to $a_j = b_j, \alpha_j = \beta_j$ for all $i$. Thus $s= \sum_j s_j a_j + \sigma_j \alpha_j$, proving that the $s_j, \sigma_j$ span $V^\vee$ over $\mathcal{O}_B$. The $\mathcal{O}_B$-linear independence of the $z_j + \theta \zeta_j, \zeta_j + \theta z_j$ follows immediately from that of the $z_j, \zeta_j$.

Let $U_i$ be one of the affine open cells covering $\mathbb{P}_{\Pi, B}(V)$. We claim that $s_i, \sigma_i$ span $\mathcal{O}_\Pi(1)(U_i)$ over $\mathcal{O}_{\mathbb{P}_\Pi(V)}(U_i)$. From the identities:

\begin{align*}
&z_j + \theta \zeta_j = (z_i + \theta \zeta_i)[(z_i + \theta \zeta_i)^{-1} (z_j + \theta \zeta_j)]\\
&\zeta_j + \theta z_j = (z_i + \theta \zeta_i)[(z_i + \theta \zeta_i)^{-1} (\zeta_j + \theta z_j)]
\end{align*}

\noindent one sees that:

\begin{align*}
s_j = s_i \bigg( \frac {z_j} {z_i} - \frac {\zeta_i \zeta_j} {z^2_i} \bigg) + \sigma_i \bigg(\frac {\zeta_j} {z_i} - \frac {\zeta_i z_j} {z^2_i}\bigg)\\
\sigma_j = s_i \bigg(\frac {\zeta_j} {z_i} - \frac {\zeta_i z_j} {z^2_i}\bigg) + \sigma_i \bigg( \frac {z_j} {z_i} - \frac {\zeta_i \zeta_j} {z^2_i} \bigg)
\end{align*}

\bigskip

\noindent for any $j \neq i$, and since we have shown that $s_j, \sigma_j$, $j = 0, \dotsc, n$ span $V^\vee$ over $\mathcal{O}_B$, we have proven the claim.

Further, we claim $s_i, \sigma_i$ are $\mathcal{O}_{\mathbb{P}_{\Pi, B}(V)}$-independent on $U_i$. For suppose $s_i \cdot a + \sigma_i \cdot \alpha = 0$ for some $a, \alpha \in \mathcal{O}_{\mathbb{P}_\Pi(V)}(U_i)$. This is equivalent to the system of equations:

\begin{equation} \label{eqn:linindep}
\begin{cases}
z_i a + \zeta_i \alpha = 0\\
\zeta_i a + z_i \alpha = 0. 
\end{cases}
\end{equation}

\medskip

Since $z_i$ is invertible on $U_i$, we see that $a = -\zeta_i \alpha/ z_i$ from \ref{eqn:linindep}. Substituting this expression for $a$ into \ref{eqn:linindep}, we find that $z_i \alpha = 0$, but by invertibility of $z_i$, $\alpha = 0$. Consequently $a = 0$ as well.

We have shown that $s_i, \sigma_i$ form a basis of $\mathcal{O}_\Pi(1)(U_i)$, hence $\mathcal{O}_\Pi(1)$ is a locally free rank $1|1$ sheaf.
\end{proof}

\medskip

\noindent {\bf Remark.} As a consequence of this proof, we obtain a particularly nice trivialization of $\mathcal{O}_\Pi(1)$ as a $\Pi$-invertible sheaf. In each $U_i$, $s_i, \sigma_i$ form a $\Pi$-symmetric basis, and

\begin{align*}
s_j &= s_i \cdot w^j_i + \sigma_i \cdot \eta^j_i\\
\sigma_j &= s_i \cdot \eta^j_i + \sigma_i \cdot w^j_i
\end{align*}

\bigskip

We thus see that the transition functions for $\mathcal{O}_\Pi(1)(U_i \cap U_j)$ are the matrix:

\begin{align*}
\left( \begin{array}{c|c}
w^j_i & \eta^j_i \\
\hline
\eta^j_i & w^j_i
\end{array}
\right)
\end{align*} 

\bigskip

We note that this matrix lies in $\mathbb{G}^{1|1}_m(U_i \cap U_j)$, as it must. 

Now we characterize the global sections of $\mathcal{O}_\Pi(1)$, assuming $B$ is affine and $V$ is trivial on $B$:

\begin{prop}
If $B = Spec(A)$ is an affine $k$-superscheme, and $V$ a free $\mathbb{D}_A$-bimodule, then $H^0(\mathbb{P}_\Pi(\widetilde{V}), \mathcal{O}_\Pi(1)) = V^\vee$.
\end{prop}

\begin{proof}
By definition $H^0(\mathbb{P}_\Pi(\widetilde{V}), \mathcal{O}_\Pi(1)) = V^\vee \otimes \Gamma(\mathcal{O}_{\mathbb{P}_\Pi(V)})$. So we only need show that $\Gamma(\mathcal{O}_{\mathbb{P}_\Pi(V)}) = A$.

First we consider the case where the $\mathbb{D}$-rank of $V$ is larger than $1$. We claim any function on $\underline{V} \backslash \{0\}$ extends uniquely to $\underline{V}$, thus is the restriction of a unique polynomial on $\underline{V}$. This should follow from super analogues of standard Hartogs'-lemma-like results in algebraic geometry, which we shall neither attempt to formulate nor prove. Instead, we give a direct argument.

Let $f$ be a function on $\underline{V} \backslash \{0\}$. The open affine subsets $U_i = \{ z_i \neq 0 \}$ cover $\underline{V} \backslash \{0\}$. Then $f|_{U_i} = P_i/z_i^{k_i}$ where $P_i$ is a polynomial, $k_i \geq 0$; we may assume for all $i$ that $z_i$ does not divide $P_i$. On the intersection $U_i \cap U_j$, $f_i|_{U_j} = f_j|_{U_i}$ if and only if $z_i^{k_i}P_j = z_j^{k_j}P_i$ in the polynomial ring $\underline{Sym}(V^*)$. If $k_i  > 0$,  we see from this equation that $z_j^{k_j}P_i$, hence $P_i$, is divisible by $z_i$. This contradicts the assumption that $z_i$ does not divide $P_i$. Hence $k_i = 0$. By the same argument $k_j = 0$, so $f|_{U_i} = P_i, f|_{U_j} = P_j$. Hence $P_i = P_j$ for all $i, j$. We conclude that $f$ extends to a polynomial on $\underline{V}$, given by $P_i$ for any $i$. This proves the uniqueness as well.

It is routine to check that any $\mathbb{D}^*$-invariant polynomial on $\underline{V}$ is in fact constant (indeed, it is enough to consider the action of the even subsupergroup $\mathbb{G}^{1|0}_m$.) Hence the proposition is proven in this case.

If the $\mathbb{D}$-rank of $V$ is $1$, $\underline{V}\backslash\{0\}$ is an affine supervariety with coordinate ring $A[z, z^{-1}, \zeta]$, on which $\mathbb{D}^*$ acts by the formula given in equation \ref{eqn:actionincoordinates}. We leave it to the reader to show, by direct calculation, that any $\mathbb{D}^*$-invariant Laurent polynomial in $z, \zeta$ is in fact constant.
\end{proof}

\subsection{Morphisms into $\mathbb{P}^n_{\Pi, B}$.}

We continue to assume that $B$ is an affine $k$-superscheme and that $E$ is a trivial sheaf of $\mathbb{D}$-modules on $B$. We have the following characterization of morphisms into $\mathbb{P}^n_{\Pi, B}$. 

\begin{thm}\label{thm:morphismsaff}
Let $B$ be an affine $k$-superscheme, and let $X \to B$ be a $B$-superscheme. If $f: X \to \mathbb{P}^n_{\Pi, B}$ is a $B$-morphism, $(f^*(\mathcal{O}_\Pi(1)), f^*(\Phi))$ is a $\Pi$-invertible sheaf on $X$, and the global sections $f^*(z_i + \zeta_i \theta), f^*(-\zeta_i + z_i \theta)$ globally generate $f^*(\mathcal{O}_\Pi(1))$. Conversely, given a $\Pi$-invertible sheaf $(S, \phi)$ on $X \to B$ and a $\Pi$-symmetric set of global sections $\{s_0, \dotsc, s_n | \\ \sigma_0, \dotsc, \sigma_n\}$ of $S$ which globally generate $S$, there exists a unique $B$-morphism $f: X \to \mathbb{P}^n_{\Pi, B}$ such that $(f^*(\mathcal{O}_\Pi(1)), f^*(\Phi)) \cong (S, \phi)$ and $f^*(z_i + \zeta_i \theta) = s_i, f^*(-\zeta_i + z_i \theta) = \sigma_i$.
\end{thm}

\begin{proof} Suppose $f: X \to \mathbb{P}^n_{\Pi, B}$ is a $B$-morphism. Then $(f^*(O_\Pi(1), f^*(\Phi))$ is a $\Pi$-invertible sheaf.

Conversely, suppose given a $\Pi$-invertible sheaf $(S, \phi)$ and a set of sections $\{s_i | \sigma_i\}_{i =0, \dotsc n}$ as given above. Clearly, the $\sigma_i$ can be recovered from the $s_i$ via $\phi$.

Let $X_i$ denote the open subset of $|X|$:

\begin{align*}
X_i = \{x \in |X| : (s_i)_x \notin \mathfrak{M}_x S_x\}.
\end{align*}

\medskip

This is an open subsuperscheme of $X$, which we also denote by $X_i$. By the hypothesis that the $s_i, t_i$ generate $S$, the $X_i$ form a cover of $X$.

Let $V$ be an open subsuperscheme of $X_i$ such that $\mathcal{O}_S$ is trivial on $V$, and let $\{e | f\}$ denote a $\Pi$-symmetric basis of $\mathcal{O}_S(V)$. Suppose that in this basis, $s_i, s_j$ are given by:

\begin{align*}
&s_i =  e a_i + f \alpha_i\\
&s_j =  e a_j + f \alpha_j.
\end{align*}

\medskip

We now define an (even) local section $s_i^{-1} \cdot s_j$ of $\mathcal{O}_X^{1|1}$ over $X_i$ as follows. Identifying $\mathcal{O}_S(V)$ with $\mathbb{D}(V)$ via the isomorphism $e \mapsto 1, f \mapsto \theta$, we can identify $s_i, s_j$ with sections $\tilde{s}_i := a_i + \theta \alpha_i, \tilde{s}_j := a_j + \theta \alpha_j$ of $\mathbb{D}(V)$. Now we make use of the arithmetic operations in the super skew algebra $\underline{\mathbb{D}}(V)$:

\begin{align*}
\tilde{s}_i^{-1} \tilde{s}_j &= (a_i + \theta \cdot \alpha_i)^{-1}(a_j + \theta \cdot \alpha_j)\\
&= \left[ \frac {a_j} {a_i} - \frac {\alpha_i \alpha_j} {a_i^2} \right] + \theta \left[\frac {\alpha_j} {a_i} - \frac {a_j \alpha_i} {a_i^2} \right]
\end{align*}

\bigskip

\noindent and then take the coefficients of $1$ and $\theta$ respectively as the components of $s_i^{-1} s_j$:

\begin{align*}
s_i^{-1}  s_j := &\left( \begin{array}{c}
\frac {a_j} {a_i} - \frac {\alpha_i \alpha_j} {a_i^2} \\
\hline
\frac {\alpha_j} {a_i} - \frac {a_j \alpha_i} {a_i^2}
\end{array} \right).
\end{align*}

\medskip

All functions involved are regular, since $a_i$ is invertible in $X_i$ by hypothesis. Now we check that $s_i^{-1}  s_j$ is independent of the $\Pi$-symmetric basis chosen and hence is well-defined. Suppose we have a change of $\Pi$-symmetric basis:

\begin{align*}
(e \, | f) = (e' | f')
\left( \begin{array}{c|c}
b & \beta\\
\hline
\beta & b
\end{array}
\right).
\end{align*}

\medskip

Then, identifying $\mathcal{O}_S(V)$ with $\underline{\mathbb{D}}(V)$ in this new basis, we have $\tilde{s}'_i = (b + \theta \beta) (a_j + \theta \alpha_j),  \tilde{s}'_j = (b + \theta \beta) (a_j + \theta \alpha_j)$, from which it follows that:

\begin{align*}
(\tilde{s}_i')^{-1} \cdot \tilde{s}'_j &= (a_i + \theta \alpha_i)^{-1} (b + \theta \beta)^{-1} (b + \theta \beta)(a_j + \theta \alpha_j)\\
& =\tilde{s}_i^{-1} \cdot \tilde{s}_j 
\end{align*}

\medskip

\noindent hence $(s'_i)^{-1} s'_j = s_i^{-1} s_j$. As the functions $w^j_i, \eta^j_i$ freely generate the $A$-superalgebra $\mathcal{O}_{U_i}$, we have a well-defined $A$-morphism  $f_i: X_i \to U_i$, given by setting:

\begin{align*}
&f^*_i(w^j_i) = (s_i^{-1} s_j)_0\\
&f^*_i(\eta^j_i) = (s_i^{-1} s_j)_1\\
\end{align*}

One may verify by direct calculation that the following four equalities hold in $\mathcal{O}_{\mathbb{P}^n_{\Pi, B}}(U_i \cap U_j)$:

\medskip

\begin{enumerate}
\item $w^i_j = (w^j_i)^{-1}$
\item $\eta^i_j = -\eta^j_i (w^j_i)^{-2}$
\item $w^k_j = w^k_i w^i_j - \eta^k_i \eta^i_j$  \hspace{2.2mm}  (for $k \neq i$)
\item $\eta^k_j = w^i_j \eta^k_i + \eta^i_j w^k_i$ \hspace{3mm}  (for $k \neq i$)
\end{enumerate}

\bigskip

The verification that the functions $(s_i^{-1} s_j)_0, (s_i^{-1} s_j)_1$ also satisfy the equalities 1) -- 4) is a completely formal matter of replacing $z_i$ with $a_i$, $\zeta_i$ with $\alpha_i$ etc. in the calculations just given. Hence it follows that $f_i|_{U_i \cap U_j} = f_j|_{U_i \cap U_j}$, and by standard arguments, the morphisms $\{f_i\}$ glue together into a morphism $f: X \to \mathbb{P}^n_\Pi$. 

One sees that by the construction of $f$, $f^*(g_{ij}) = h_{ij}$, where $g_{ij}$ are the transition functions of $\mathcal{O}_\Pi$ computed in the previous remark, and $h_{ij}$ are the transition functions of $S$ on the cover $X_i$. Hence $f^*(\mathcal{O}_\Pi(1)) \cong S$. Similarly, one checks immediately that $f^*(z_j + \zeta_j \theta) = s_j, f^*(-\zeta_j + z_j \theta) = \sigma_j$, and that $f^*(\Phi) = \phi$ (the last follows from the $\Pi$-symmetry of $s_i, t_i$ for all $i$).

The uniqueness statement in the proposition follows, since any morphism $f': X \to \mathbb{P}^n_{\Pi, B}$ which satisfies the conditions of the theorem must agree with $f$ on each $U_i$, hence must equal $f$.

\end{proof}

\section{Product structure on $\Pi$-invertible sheaves} Working in the category of complex supermanifolds, Voronov Manin and Penkov \cite{VMP} define a notion of a {\it composition} of an ordered pair of $\Pi$-invertible sheaves. The result of this composition is not a $\Pi$-invertible sheaf, but rather a $1|1$ locally free sheaf, and its significance is therefore somewhat obscured. We shall clarify matters by using the algebra of the super skew field $\mathbb{D}$ to define a product operation on ordered pairs of $\Pi$-invertible sheaves (which takes values in $1|1$ locally free sheaves), and then showing that our product is the same as the composition of Voronov, Manin, and Penkov.

\noindent 

\subsection{Super Morita theory} Let $A$, $B$ be super rings with unit (not necessarily supercommutative). Note that an $(A,B)$-bimodule is the same thing as a left $A \otimes B^o$-module by the following recipe: if $M$ is a left $A \otimes B^o$ module, we define an $(A, B)$ bimodule structure on $M$ by:

\begin{align*}
a  \cdot_A \, m \cdot_B b := (-1)^{|b||m|}(a \otimes b) \cdot m
\end{align*}

\medskip

where $a \in A$, $b \in B$, and $m \in M$ are all homogeneous. Conversely, if $M$ is an $(A, B)$-bimodule, we may define a left $A \otimes B^o$-module structure on $M$ using the same formula. It is readily seen that these correspondences are compatible with morphisms in the respective categories, hence define a category equivalence between the category of $(A,B)$-bimodules and that of left $A \otimes B^o$-modules. 

From now on, we assume that $A$ is an $R$-superalgebra, with $R$ supercommutative.

\begin{defn}
Let $M$ be an $(A,A)$-bimodule. The {\it supercommutant} of $M$ is the $R$-module $M^A$ generated by the set:

\begin{equation*}
\{m \in M : a m = (-1)^{|a||m|} ma, m \text{ homogeneous} \}
\end{equation*}
\end{defn}

\medskip

Equivalently, interpreting $M$ as a left $A \otimes A^o$-module, we see that $M^A$ may be defined in terms of the $A \otimes A^o$-action as the $R$-module $M^A$ generated by the set:

\begin{equation*}
\{m \in M : (a \otimes 1) \cdot m = (1 \otimes a) \cdot m, m \text{ homogeneous}\}
\end{equation*}

\medskip

For brevity we will denote the superalgebra $A \otimes_R A^o$ by $A^e$.

We will need the following theorem from the Morita theory of super rings, proven in \cite{Kwok}.

\begin{thm}\label{thm:supermorita}
Let $R$ be a supercommutative ring, and suppose $A$ is a super Azumaya algebra over $R$. Then $V \mapsto A \otimes_R V: \mathfrak{M}_R \to {}_A\mathfrak{M}_A$ and $W \mapsto W^A: {}_A \mathfrak{M}_A \to {}_R\mathfrak{M}$ are mutually inverse category equivalences.
\end{thm}

\subsection{Definition of the product}

Let $A/R$ be an $R$-superalgebra, $R$ supercommutative. We begin by noting that if $M$ is a $A$-module, $N$ a right $A$-module, then $M \otimes_R N$ is a $(A, A)$ bimodule via the formula:

\begin{align*}
a_1 \cdot (m \otimes n) \cdot a_2 := (a_1 \cdot m) \otimes (n \cdot a_2).
\end{align*}

\medskip

Theorem \ref{thm:supermorita} tells us that, given a sheaf of $(\underline{\mathbb{D}}_X, \underline{\mathbb{D}}_X)$-bimodules $E$, there corresponds in a natural way a sheaf of $\mathcal{O}_X$-modules given by the supercommutant sheaf $E^{\underline{\mathbb{D}}_X}$. We shall define our product via this correspondence. We begin with the following


\begin{prop}\label{supercommutantfree}
Let $R$ be a commutative $k$-superalgebra, $M$ a free left $\mathbb{D}_R$-module of rank $1$, and $N$ be a free right $\mathbb{D}_R$-module of rank $1$ (hence $M, N$ are free $R$-modules of rank $1|1$). Then the supercommutant $(M \otimes_R N)^{\mathbb{D}_R}$ is a free $R$-module of rank $1|1$.
\end{prop}

\begin{proof}

By the previous lemma, there exist $R$-module bases $\{e | f\}, \{e' | f'\}$ of $M, N$ respectively such that:

\begin{align*}
&\theta \cdot e = f\\
&\theta \cdot f = -e\\
&e' \cdot \theta = f'\\
&f' \cdot \theta = -e'\\
\end{align*}

Then $\mathcal{B} := \{e \otimes e', f \otimes f' | e \otimes f', f \otimes e'\}$ is an $R$-module basis of $M \otimes N$. We will now compute the homogeneous elements of the supercommutant. For now, suppose $w \in (M \otimes N)^{\mathbb{D}_R}$ is even. Then

\begin{align*}
w = (e \otimes e') a + (f \otimes f') b + (e \otimes f') \alpha + (f \otimes e') \beta
\end{align*}

\medskip

\noindent where $a, b, \alpha, \beta$ are uniquely determined elements of $R$ such that $a, b$ (resp. $\alpha, \beta$) are even (resp. odd).

The assertion that $w \in (M \otimes N)^{\mathbb{D}_R}$ is the same as the equality:

\begin{equation}\label{eq: commutant}
\theta \cdot w = w \cdot \theta.
\end{equation}

\medskip

One checks by direct calculation that (\ref{eq: commutant}) holds if and only if $b = -a, \beta = -\alpha$, so that $w = (e \otimes e' - f \otimes f') a + (e \otimes f' - f \otimes e') \alpha$. Let us define $u := e \otimes e' - f \otimes f', v := e \otimes f' - f \otimes e'$.

Then $w = u \cdot a + v \cdot \alpha$, so that any even $w \in (M \otimes N)^{\mathbb{D}_R}$ is an $R$-linear combination of $u$ and $v$, with $a, \alpha$ uniquely determined. By a completely analogous argument we see that for odd $w$, $w = u \cdot \alpha + v \cdot a$, for uniquely determined $a, \alpha$. Hence $\{u | v\}$ form a homogeneous basis of $(M \otimes N)^{\mathbb{D}_R}$, and $(M \otimes N)^{\mathbb{D}_R}$ is a free $R$-module of rank $1|1$. 

\end{proof}

\noindent {\bf Remark.} A similar proposition can easily be proven for the tensor product $M'\otimes N'$ of free rank 1 left (resp. right) $\mathbb{D}^o_R$-modules $M'$ and $N'$; the arguments are essentially the same as the above.\\

Now we may define our products. Let us choose a $\sqrt{-1}$ in $k$. Then given an ordered pair of $\Pi$-invertible sheaves $(S, \phi)$ and $(S', \phi')$ on a $B$-superscheme $X$, we may form two canonically defined (up to our choice of $\sqrt{-1}$) sheaves of $\mathcal{O}_X$-modules, denoted by $S \boxtimes S'$ and $S \boxtimes_o S'$, as follows.

To define $S \boxtimes S$, $(S, \sqrt{-1} \phi)$ is regarded as a sheaf of left $\underline{\mathbb{D}}$-modules, $(S, \phi')$ as a sheaf of right $\underline{\mathbb{D}}$-modules, so that $S \otimes S'$ is a sheaf of $(\underline{\mathbb{D}}, \underline{\mathbb{D}})$-bimodules. Then we define:

\begin{equation*}
S \boxtimes S' := (S \otimes S')^{\underline{\mathbb{D}}}.
\end{equation*}

\medskip

\noindent More explicitly, for each open set $U$, $(S \boxtimes S')(U)$ is the $\mathcal{O}_U$-module generated by:

\begin{align*}
&\{s \otimes s' \in \mathcal{O}_{S \otimes S'}(U): \sqrt{-1} \phi(s) \otimes s' =  (-1)^{|s|} s \otimes \phi'(s'), s \in \mathcal{O}_S(U), \\ &s' \in \mathcal{O}_{S'}(U), s, s' \text{ homogeneous} \}.
\end{align*}

\medskip

It is routine to check this is a sheaf, since $\phi, \phi'$ are global endomorphisms. Applying Prop \ref{supercommutantfree} to sufficiently small open sets $U$, we see that $S \boxtimes S'$ so defined is a locally free sheaf of $\mathcal{O}_X$-modules of rank $1|1$. 

To define $S \boxtimes_o S'$, we instead regard $(S, \phi)$ as a sheaf of left $\underline{\mathbb{D}}^o$-modules and $(S', \sqrt{-1} \phi')$ as a sheaf of right $\underline{\mathbb{D}}^o$-modules; then $(S \otimes S', \phi, \sqrt{-1} \phi')$ is a sheaf of $(\underline{\mathbb{D}}^o, \underline{\mathbb{D}}^o)$-bimodules, and we define $S \boxtimes_o S' := (S \otimes S')^{\underline{\mathbb{D}}^o}$. For each open set $U$, $(S \otimes S')^{\underline{\mathbb{D}}^o}(U)$ is the $\mathcal{O}_U$-module generated by:

\begin{align*}
&\{s \otimes s' \in \mathcal{O}_{S \otimes S'}(U): \phi(s) \otimes s' =  \sqrt{-1}(-1)^{|s|} s \otimes \phi'(s'), s \in \mathcal{O}_S(U), \\ &s' \in \mathcal{O}_{S'}(U), s, s' \text{ homogeneous} \}.
\end{align*}

By the remark following Prop. \ref{supercommutantfree}, we may apply the $\mathbb{D}^o$-analogue of Prop. \ref{supercommutantfree} to show that $S \boxtimes_o S'$ is also a locally free sheaf of $\mathcal{O}_X$-modules of rank $1|1$.

\subsection{Equivalence with the composition of Voronov, Manin, and Penkov} 

In the category of complex supermanifolds, Voronov, Manin, and Penkov \cite{VMP} define the composition of two $\Pi$-invertible sheaves $(S, \phi),  (S', \phi')$ as follows: fix a $\sqrt{-1}$. Then $\phi \otimes \phi'$ is an even endomorphism of square $-1$ on $S \otimes S'$. The eigenspaces for $\phi \otimes \phi'$, which necessarily have eigenvalues $\pm \sqrt{-1}$, are what they call the result of the composition of $(S, \phi)$ and $(S', \phi')$. These eigenspaces are $1|1$ locally free sheaves.

This definition can be carried over to the category of $B$-superschemes without change. In this context, we now demonstrate the equivalence of their definitions with our products $\boxtimes$ and $\boxtimes_o$.

Regarding $S \otimes S'$ as a sheaf of $(\underline{\mathbb{D}}, \underline{\mathbb{D}})$-bimodules via $\sqrt{-1}\phi$ and $\phi'$, we claim that $S \boxtimes S'$ equals the $\sqrt{-1}$ eigenspace of $\phi \otimes \phi'$. For if $s \otimes s'$ is a basic element of $S \otimes S'$, we have:

\begin{align*}
&\sqrt{-1} (\phi (s) \otimes s') = (-1)^{|s|} (s \otimes \phi'(s'))\\
\iff &\sqrt{-1} (\phi^2(s) \otimes s') = (-1)^{|s|} (\phi \otimes 1) \cdot (s \otimes \phi'(s')) \\
\iff & \sqrt{-1} (s \otimes s') = (\phi \otimes \phi') \cdot (s \otimes s').
\end{align*}

\medskip

\noindent and the same is true of linear combinations of basic elements. 

Similarly, if we regard $S \otimes S'$ as a sheaf of $(\underline{\mathbb{D}}^o, \underline{\mathbb{D}}^o)$-bimodules, via $\phi$ and $\sqrt{-1} \phi'$, the $-\sqrt{-1}$-eigenspace of $\phi \otimes \phi'$ equals the product $S\boxtimes_o S'$; the arguments are entirely analogous to the ones just given.

\end{document}